\documentclass[reqno,oneside,12pt]{elsarticle}
\usepackage{fancybox,ascmac,comment}
\usepackage{amssymb}
\usepackage{amsfonts,mathrsfs}
\usepackage{amsmath}
\usepackage{amsthm}
\usepackage{mathtools}
\usepackage{enumerate}
\usepackage{mleftright}
\mleftright
\mathtoolsset{showonlyrefs=true}
\numberwithin{equation}{section}
\makeatletter
\newcommand{\shortmathcal}[1]{\@tfor\ch:=#1\do{
\expandafter\edef\csname c\ch\endcsname{\noexpand\mathcal{\ch}}
}}
\newcommand{\shortmathbb}[1]{\@tfor\ch:=#1\do{
\expandafter\edef\csname bb\ch\endcsname{\noexpand\mathbb{\ch}}
}}
\newcommand{\shortmathbf}[1]{\@tfor\ch:=#1\do{
\expandafter\edef\csname b\ch\endcsname{\noexpand\mathbf{\ch}}
}}
\newcommand{\shortboldsymbol}[1]{\@tfor\ch:=#1\do{
\expandafter\edef\csname bs\ch\endcsname{\noexpand\boldsymbol{\ch}}
}}
\newcommand{\shortmathfrak}[1]{\@tfor\ch:=#1\do{
\expandafter\edef\csname f\ch\endcsname{\noexpand\mathfrak{\ch}}}}
\newcommand{\shortmathscr}[1]{\@tfor\ch:=#1\do{
\expandafter\edef\csname s\ch\endcsname{\noexpand\mathscr{\ch}}}}
\newcommand{\shortmathrm}[1]{\@tfor\ch:=#1\do{
\expandafter\edef\csname r\ch\endcsname{\noexpand\mathrm{\ch}}
}}
\makeatother
\shortmathcal{ABCDEFGHIJKLMNOPQRSTUVWXYZ}
\shortmathbb{ABCDEFGHIJKLMNOPQRSTUVWXYZ}
\shortmathbf{ABCDEFGHIJKLMNOPQRSTUVWXYZabcdefghijklmnopqrstuvwxyz}
\shortboldsymbol{ABCDEFGHIJKLMNOPQRSTUVWXYZabcdefghijklmnopqrstuvwxyz}
\shortmathfrak{ABCDEFGHIJKLMNOPQRSTUVWXYZabcdefghjklmnopqrstuvwxyz}
\shortmathscr{ABCDEFGHIJKLMNOPQRSTUVWXYZ}
\shortmathrm{ABCDEFGHIJKLMNOPQRSTUVWXYZabcdefghijklmnopqrstuvwxyz}
\newcommand{\AP}{\mathsf{AP}}

\newcommand{\jump}[1]{\ensuremath{[\![#1]\!]}}
\newcommand{\ncjump}[1]{\ensuremath{\langle\!\langle #1\rangle\!\rangle}}

\newcommand{\wt}{\mathrm{wt}}

\newtheorem{theorem}{Theorem}[section]

\newtheorem{lemma}[theorem]{Lemma}

\theoremstyle{definition}
\newtheorem{definition}[theorem]{Definition}
\theoremstyle{remark}
\newtheorem{remark}[theorem]{Remark}

\allowdisplaybreaks
\title{Formal sine functions in harmonic algebra}
\author[1]{Hanamichi Kawamura\corref{me}
\fnref{email}}
\ead{1121026@ed.tus.ac.jp}
\cortext[me]{Corresponding author}
\fntext[email]{\emph{Email address:} \texttt{1125512@ed.tus.ac.jp}.}
\affiliation[1]{organization={Department of Mathematics, Faculty of Science Division I, Tokyo University of Science},
addressline={1-3 Kagurazaka, Shinjuku-ku},
postcode={162-8601},
city={Tokyo},
country={Japan}}
\begin{document}
\begin{abstract}
    In this paper, we introduce formal sine functions whose coefficients are elements of a generalized harmonic algebra and investigate their properties corresponding to the classical addition formula and Pythagorean theorem.
    By taking their image under a linear map with some conditions, they coincide with the classical sine function.
    Moreover, we show that one such linear map is obtained from iterated integrals.
\end{abstract}
\maketitle
\section{Introduction}\label{sec:introduction}
Both the following theorems are most celebrated and classical formulas for trigonometric functions:
\begin{theorem}[Addition formula]\label{thm:classical_af}
The identity $\sin(x+y)=\sin x\cos y+\cos x\sin y$ holds for $x,y\in\bbR$.
\end{theorem}
\begin{theorem}[Pythagorean theorem]\label{thm:classical_pt}
The identity $\sin^{2}x+\cos^{2}x=1$ holds for $x\in\bbR$.
\end{theorem}
Considering $\cos$ as the derivative of $\sin$, these formulas are functional relations for the sine function.
Moreover, the sine function has two expressions that known as the \emph{Taylor series} 
\begin{equation}\label{eq:sin_series}
    \frac{\sin (\pi x)}{\pi}=\sum_{n=0}^{\infty}(-1)^{n}\frac{\pi^{2n}}{(2n+1)!}x^{2n+1}
\end{equation}
and as the \emph{reflection formula}
\begin{equation}\label{eq:sin_reflection}
    \frac{\sin (\pi x)}{\pi}=\frac{x}{\Gamma(x)\Gamma(1-x)},
\end{equation}
respectively, where $\Gamma(x)$ denotes the gamma function.
The ingredients of these expressions are viewed as \emph{multiple zeta values}; indeed, $\pi^{2r}/(2r+1)!$ coincides with the special case $(k_{1},\ldots,k_{r})=(2,\ldots,2)$ of the multiple zeta value
\[\zeta(k_{1},\ldots,k_{r})\coloneqq\sum_{0<n_{1}<\cdots<n_{r}}\frac{1}{n_{1}^{k_{1}}\cdots n_{r}^{k_{r}}},\qquad (k_{1},\ldots,k_{r-1}\ge 1,~k_{r}\ge 2)\]
and the logarithm of the gamma function essentially gives a generating function of Riemann zeta values.
From these facts, we can view both \eqref{eq:sin_series} and \eqref{eq:sin_reflection} as expressions of $\sin(\pi x)/\pi$ as a power series whose coefficients are multiple zeta values.
In this paper, we define two types of a \emph{formal sine function}, $\sS_{z,k}^{\rR}(z)$ and $\sS_{z,k}^{\rT}(z)$, in a generalization of Hoffman's \emph{harmonic algebra} and investigate their properties: the addition formula and Pythagorean theorem.
Since multiple zeta values are obtained by computing \emph{iterated integrals} for elements of harmonic algebra, we can recover such classical formulas for the sine function from those properties.
Our main theorem is the following (we give the definitions of some symbols later. Especially, for the element $w_{z,n}$, the formal sine function $\sS_{z}(x)$ and formal cosine function $\sC_{z}(x)$, see the beginning of Section \ref{sec:af}).
\begin{theorem}\label{thm:main}
    Let $M$ be a monoid with a zero element, $z$ its element and $\AP_{z}$ the ideal of the harmonic algebra $\fH_{M,\ast}$ (with respect to $M$) generated by elements of the form $w_{z,m+n}-w_{z,m}\ast w_{z,n}$ with $m,n\in\bbZ_{\ge 0}$.
    Then, for an ideal $I\subseteq\fH_{M,\ast}$, the following are equivalent.
    \begin{enumerate}[\upshape (i)]
        \item $I$ is $\AP_{z}$-ample; that is, $\AP_{z}\subseteq I$.
        \item A modified addition formula
        \[\sS_{z}(x+y)=\sS_{z}(x)\ast\sC_{z}(y)+\sC_{z}(x)\ast \sS_{z}(y)\]
        holds in $(\fH_{M,\ast}/I)\jump{x,y}$.
        \item A modified Pythagorean theorem
        \[-w_{z,1}\ast\sS_{z}(x)^{\ast 2}+\sC_{z}(x)^{\ast 2}=1\]
        holds in $(\fH_{M,\ast}/I)\jump{x}$.
    \end{enumerate}
\end{theorem}
Moreover, since the kernel of the homomorphism $Z$ given by iterated integrals (constructed in Section \ref{sec:ii}) is $\AP_{1}$-ample, we have the addition formula and Pythagorean theorem for $Z(\sS_{1}(x))$, which realize classical formulas (Theorems \ref{thm:classical_af} and \ref{thm:classical_pt}).
\section{Formal sine functions based on Taylor series and reflection formula}
\subsection{Algebraic formulation}
From here to the end of Section \ref{sec:pt}, fix a monoid $M$ having a zero element.
The identity element (resp.~zero element) of $M$ is denoted by $1$ (resp.~$0$).
Denote by $\fH_{M}$ the non-commutative polynomial ring $\bbQ\langle e_{z}\mid z\in M\rangle$ generated by formal symbols $\{e_{z}\}_{z\in M}$ corresponding to elements of $M$ over $\bbQ$.
Let us endow $\fH_{M}$ a commutative and associative $\bbQ$-algebra structure.
Define the \emph{harmonic product} $\ast$ by the $\bbQ$-bilinearity and the rules $1\ast w=w\ast 1=w$ ($w\in\fH_{M}$) and
\begin{equation}\label{eq:definition_harmonic}
    e_{a}w\ast e_{b}w'=e_{ab}(w\ast e_{b}w'+e_{a}w\ast w'-e_{0}(w\ast w')),
\end{equation}
for $a,b\in M$ and $w,w'\in\fH_{M}$.
Hereafter, when we encounter a notation denoting a product with $\ast$ and the usual concatenation product together without parentheses, like $w\ast e_{b}w'$ above, the harmonic product is always prior to the concatenation.
\begin{remark}
This definition is introduced by Hirose--Sato \cite{hs20} as a generalization of the original harmonic product \cite[Section 2]{hoffman97}.
For a proof of the commutativity and associativity of $\ast$, see \cite[Proposition 5 (ii)]{hs20}.
\end{remark}
When we consider $\fH_{M}$ as a commutative associative $\bbQ$-algebra equipped with $\ast$, denote it by $\fH_{M,\ast}$.
We sometimes mention $\fH_{M,\ast}$ as the \emph{harmonic algebra}.
\begin{lemma}\label{lem:harmonic_e0}
    For every $w\in\fH_{M}$ and a non-negative integer $m$, we have $e_{0}^{m}\ast w=e_{0}^{m}w$.
\end{lemma}
\begin{proof}
    By the linearity, it is sufficient to give a proof for the case where $m=1$ and $w$ is a word $e_{a_{1}}\cdots e_{a_{k}}$.
    We prove it by induction on $k$.
    The $k=0$ case is clear from the definition.
    When it holds for some $k\ge 0$, for $a\in M$ we see that
    \[
    e_{0}\ast e_{a}w
    =e_{0}(1\ast e_{a}w+e_{0}\ast w-e_{0}(1\ast w))\\
    =e_{0}e_{a}w+e_{0}(e_{0}\ast w-e_{0}w),
    \]
    and the last term vanishes by the induction hypothesis.
\end{proof}
\subsection{Definitions of formal sine functions}
First we define the formal sine function based on the Taylor series and reflection formula.
For $z\in M$ and a positive integer $k$, put $s_{z,k}\coloneqq e_{z}e_{0}^{k-1}$.
Using this notation, we see that the rule of the harmonic product \eqref{eq:definition_harmonic} is equivalent to
\begin{equation}\label{eq:usual_harmonic}
    s_{a,k}w\ast s_{b,l}w'=s_{ab,k}(w\ast s_{b,l}w')+s_{ab,l}(s_{a,k}w\ast w')-s_{ab,k+l}(w\ast w'),
\end{equation}
by Lemma \ref{lem:harmonic_e0}.
Note that there is a similar statement in \cite[Proposition 5 (i)]{hs20}.
\begin{definition}
    For a positive integer $k$ and $z\in M$, we define two elements $\sS_{z,k}^{\rT}$ and $\sS_{z,k}^{\rR}$ of $\fH_{M,\ast}\jump{x}$ by
    \[\sS_{z,k}^{\rT}(x)\coloneqq\sum_{n=0}^{\infty}s_{z^{n},k}s_{z^{n-1},k}\cdots s_{z,k}x^{(n+1)k-1}\]
    and    
    \[\sS_{z,k}^{\rR}(x)\coloneqq x^{k-1}\exp_{\ast}\left(\sum_{n=1}^{\infty}\frac{s_{z^{n},nk}}{n}x^{nk}\right),\]
    where
    \[\exp_{\ast}\colon x\fH_{M,\ast}\jump{x}\to \fH_{M,\ast}\jump{x};~f\mapsto 1+\sum_{n>0}\underbrace{f\ast\cdots \ast f}_{n}\frac{1}{n!}
    \]
    is the exponential function with respect to the harmonic product.
\end{definition}
\begin{remark}\label{rem:assumption}
Assume there exists a $\bbQ$-algebra homomorphism $Z\colon\fH_{M,\ast}\to\bbC$ such that
\begin{enumerate}[(i)]
    \item\label{ass1} $Z(s_{1,2}^{n})=(-1)^{n}\pi^{2n}/(2n+1)!$ for every $n\ge 0$.
    \item\label{ass2} $Z(s_{1,1})=0$.
    \item\label{ass3} $Z(s_{1,k})=-\zeta(k)$ for every $k\ge 2$.
\end{enumerate}
Extend $Z$ to $\fH_{M,\ast}\jump{x}$ coefficientwise. 
Then we see that $\sin(\pi x)/\pi$ is equal to both $Z(\sS_{1,2}^{\rT}(x))$ and $Z(\sS_{1,2}^{\rR}(x))$ by \eqref{eq:sin_series} and \eqref{eq:sin_reflection}, respectively.
This is why we call $\sS_{z,k}^{\rT}$ and $\sS_{z,k}^{\rR}$ formal sine functions.
As mentioned in Section \ref{sec:introduction}, one such $Z$ is found as the map sending elements of the harmonic algebra to iterated integrals.
\end{remark}
\begin{remark}
    The symbols $\rT$ and $\rR$ stand for the `T'aylor series and `R'eflection formula, respectively.
    For the latter, indeed we can easily show that
    \[\sS_{z^{2},2}^{\rR}(x)=x\sS_{z,1}^{\rR}(x)\ast\sS_{z,1}^{\rR}(-x)\]
    and
    \[Z(\sS_{1,1}^{\rR}(x))=e^{\gamma x}\Gamma(1-x)^{-1},\]
    where $\gamma$ denotes Euler's constant.
\end{remark} 
\subsection{Coincidence of $\sS_{z,k}^{\rT}$ and $\sS_{z,k}^{\rR}$}
The following theorem states that two formal sine functions defined in the previous subsection are actually equal.
\begin{theorem}\label{thm:R_T_coincidence}
    In $\fH_{M,\ast}\jump{x}$, we have $\sS_{z,k}^{\rT}(x)=\sS_{z,k}^{\rR}(x)$ for every $z\in M$ and $k\in\bbZ_{\ge 1}$.
\end{theorem}
\begin{proof}
    The statement $\sS_{z,k}^{\rT}(x)=\sS_{z,k}^{\rR}(x)$ is equivalent to 
    \begin{equation}\label{eq:coincidence_1st}
        \sum_{n=0}^{\infty}s_{z^{n},k}\cdots s_{z,k}x^{nk}=\exp_{\ast}\left(\sum_{n=1}^{\infty}\frac{s_{z^{n},nk}}{n}x^{nk}\right).
    \end{equation}
    Since both sides above have the constant term $1$, it suffices to show $\sS_{z,k}^{\rT}(x)=\sS_{z,k}^{\rR}(x)$ after taking the logarithmic derivative, that is, we shall prove
    \begin{equation}\label{eq:coincidence_2nd}
        \left(\sum_{n=0}^{\infty}nk\cdot s_{z^{n},k}\cdots s_{z,k}x^{nk-1}\right)\ast \left(\sum_{n=0}^{\infty}s_{z^{n},k}\cdots s_{z,k}x^{nk}\right)^{\ast(-1)}=k\sum_{n=1}^{\infty}s_{z^{n},nk}x^{nk-1},
    \end{equation}
    where $\ast(-1)$ means the inverse in $\fH_{M,\ast}\jump{x}$.
    As this identity \eqref{eq:coincidence_2nd} is also equivalent to
    \begin{equation}\label{eq:coincidence_3rd}
        \sum_{n=0}^{\infty}n\cdot s_{z^{n},k}\cdots s_{z,k}x^{nk}=\left(\sum_{n=1}^{\infty}s_{z^{n},nk}x^{nk}\right)\ast\left(\sum_{n=0}^{\infty}s_{z^{n},k}\cdots s_{z,k}x^{nk}\right),
    \end{equation}
    what we prove is the coincidence of their coefficients
    \begin{equation}\label{eq:coincidence_4th}
        n\cdot s_{z^{n},k}\cdots s_{z,k}=\sum_{i=1}^{n}s_{z^{i},ik}\ast s_{z^{n-i},k}\cdots s_{z,k}
    \end{equation}
    for $n\ge 1$.
    Write the right-hand side as $Q(n)$ and let us use induction on $n$.
    The $n=1$ case is clear.
    By \eqref{eq:usual_harmonic}, we have
    \begin{align}
        &Q(n)\\
        &=s_{z^{n},nk}+\sum_{i=1}^{n-1}s_{z^{i},ik}\ast s_{z^{n-i},k}\cdots s_{z,k}\\
        &=\begin{multlined}[t]
            s_{z^{n},nk}+\sum_{i=1}^{n-1}\left(s_{z^{n},ik}(1\ast s_{z^{n-i},k}\cdots s_{z,k})+s_{z^{n},k}(s_{z^{i},ik}\ast s_{z^{n-i-1},k}\cdots s_{z,k})\right.\\
            \left.-s_{z^{n},(i+1)k}(1\ast s_{z^{n-i-1},k}\cdots s_{z,k})\right)
        \end{multlined}\\
        &=\begin{multlined}[t]
            s_{z^{n},nk}+s_{z^{n},k}Q(n-1)\\
            +\sum_{i=1}^{n-1}\left(s_{z^{n},ik}s_{z^{n-i},k}\cdots s_{z,k}-s_{z^{n},(i+1)k}s_{z^{n-i-1},k}\cdots s_{z,k}\right).
        \end{multlined}
    \end{align}
    Since the last finite sum is telescopic, we obtain
    \[Q(n)=s_{z^{n},k}Q(n-1)+s_{z^{n},k}s_{z^{n-1},k}\cdots s_{z,k}.\]
    Then, using the induction hypothesis, we get \eqref{eq:coincidence_4th}.
\end{proof}
\begin{remark}
    Theorem \ref{thm:R_T_coincidence} gives another proof of Ikeda--Sakata's theorem \cite[Theorem 2.1]{is23}.
    Indeed it follows by putting $z=1$ in Theorem \ref{thm:R_T_coincidence} and identifying $(-1)^{r}s_{1,k_{1}}\cdots s_{1,k_{r}}$ in our notation with $[k_{1},\ldots,k_{r}]$ in their notation.
\end{remark}
\section{Addition formula}\label{sec:af}
Because Theorem \ref{thm:R_T_coincidence} holds and we only consider the case where $k=2$ hereafter, we use the notation $\sS_{z}(x)\coloneqq \sS_{z,2}^{\rT}(x)=\sS_{z,2}^{\rR}(x)$.
Recall the homomorphism $Z$ of which we assumed the existence in Remark \ref{rem:assumption} and $Z(\sS_{1,2}(x))=\sin (\pi x)/\pi$.
Mimicking one of definitions of the cosine function
\[\cos(\pi x)\coloneqq \frac{d}{dx}\frac{\sin (\pi x)}{\pi},\]
we define $\sC_{z}(x)\coloneqq\sS'_{z}(x)$.
Moreover, since the addition formula (Theorem \ref{thm:classical_af}) is rewritten as
\[Z(\sS_{1}(x+y))=Z(\sS_{1}(x))Z(\sC_{1}(y))+Z(\sC_{1}(x))Z(\sS_{1}(y)),\]
we consider the power series
\[A_{z}(x,y)\coloneqq \sS_{z}(x+y)-\sS_{z}(x)\ast\sC_{z}(y)-\sC_{z}(x)\ast \sS_{z}(y)\in\fH_{M,\ast}\jump{x,y}.\]

For $z\in M$, put $w_{z,n}\coloneqq (2n+1)!s_{z^{n},2}s_{z^{n-1},2}\cdots s_{z,2}$ and $g_{z,m,n}\coloneqq w_{z,m+n}-w_{z,m}\ast w_{z,n}$ for non-negative integers $m$ and $n$.
Let $\AP_{z}$ be the ideal of $\fH_{M,\ast}$ generated by $\{g_{z,m,n}\mid m,n\ge 0\}$ and refer any ideal $I\subseteq \fH_{M,\ast}$ containing $\AP_{z}$ as an \emph{$\AP_{z}$-ample ideal}.
\begin{theorem}\label{thm:main_af}
    Let $z$ be an element of $M$ and $I$ an ideal of $\fH_{M,\ast}$.
    Then the following are equivalent.
    \begin{enumerate}[\upshape (i)]
        \item\label{af1} $I$ is $\AP_{z}$-ample.
        \item\label{af2} $A_{z}(x,y)=0$ holds in the quotient ring $(\fH_{M,\ast}/I)\jump{x,y}$.
    \end{enumerate}
\end{theorem}
\begin{proof}
    Let us denote the coefficient of a power series $\Phi(x,y)$ at $x^{m}y^{n}$ by $C(\Phi;x^{m}y^{n})$.
    Since $\sS_{z}(x)$ (resp.~$\sC_{z}(x)$) is an odd (resp.~even) function of $x$, the terms whose homogeneous degree is odd only appear in $A_{z}(x,y)$.
    Thus let us compute $C(A_{z};x^{i}y^{2N+1-i})$ for $0\le i\le N$.  
    From the explicit series expression of $\sS_{z}$ (definition as $\sS_{z,2}^{\rT}$), we obtain
    \begin{align}
        C(\sS_{z}(x+y);x^{i}y^{2N+1-i})
        &=\frac{w_{z,N}}{i!(2N+1-i)!},\\
        C(\sS_{z}(x);x^{i})
        &=\begin{cases}
        w_{z,(i-1)/2}/i! & \text{if }i\text{ is odd},\\
        0 & \text{if }i\text{ is even},
        \end{cases}
    \end{align}
    and
    \begin{align}
    C(\sC_{z}(x);x^{i})
        &=C(\sS'_{z}(x);x^{i})\\
        &=\begin{cases}
        0 & \text{if }i\text{ is odd},\\
        w_{z,i/2}/i! & \text{if }i\text{ is even}.
        \end{cases}
    \end{align}
    Therefore we have
    \begin{equation}\label{eq:af_coefficient}
        C(A_{z};x^{i}y^{2N+1-i})
        =\frac{1}{i!(2N+1-i)!}
        \cdot\begin{cases}g_{z,(i-1)/2,N-(i-1)/2} & \text{if }i\text{ is odd},\\
        g_{z,i/2,N-i/2} & \text{if }i\text{ is even}.
        \end{cases}
    \end{equation}
    Thus \eqref{af1} implies \eqref{af2}.
    Conversely, an arbitrary generator $g_{z,m,n}$ of $\AP_{z}$ has an expression
    \[g_{z,m,n}=(2m+1)!(2n)!\cdot C(A_{z};x^{2m+1}y^{2n})=(2m)!(2n+1)!\cdot C(A_{z};x^{2m}y^{2n+1}).\]
    Hence \eqref{af2} implies \eqref{af1}.
\end{proof}
\begin{remark}
    The first condition \eqref{ass1} in Remark \ref{rem:assumption} states $Z(w_{1,n})=\pi^{2n}$.
    Therefore, the kernel of the map $Z$ is automatically $\AP_{1}$-ample and thus we obtain the addition formula for $Z(\sS_{1}(x))$ by virtue of Theorem \ref{thm:main_af}.
    Furthermore, the Pythagorean theorem for $Z(\sS_{1}(x))$ also holds by Theorem \ref{thm:main_pt} below. 
\end{remark}
\section{Pythagorean theorem}\label{sec:pt}
Recall $Z(\sS_{1}(x))=\sin(\pi x)/\pi$ and $Z(\sC_{1}(x))=\cos(\pi x)$.
In terms of them, we can rephrase the classical Pythagorean theorem (Theorem \ref{thm:classical_pt}) as
\[-Z(w_{1,1})Z(\sS_{1}(x))^{2}+Z(\sC_{1}(x))^{2}=1,\]
where we have used $Z(w_{1,1})=-\pi^{2}$ which follows from the assumptions \eqref{ass1} and \eqref{ass3} in Remark \ref{rem:assumption}.
Imitating this, we define
\[P_{z}(x)\coloneqq -w_{z,1}\ast\sS_{z}(x)^{\ast 2}+\sC_{z}(x)^{\ast 2}\in\fH_{M,\ast}\jump{x},\]
where $\ast n$ denotes the $n$-th power by the harmonic product.
\begin{theorem}\label{thm:main_pt}
    Let $z$ be an element of $M$ and $I\subseteq\fH_{M,\ast}$ an ideal.
    Then the following are equivalent.
    \begin{enumerate}[\upshape (i)]
        \item\label{pt1} $I$ is $\AP_{z}$-ample.
        \item\label{pt2} $P_{z}(x)=1$ holds in the quotient ring $(\fH_{M,\ast}/I)\jump{x}$.
    \end{enumerate}
\end{theorem}  
\begin{proof}
    Obviously $P_{z}(x)$ is an even function of $z$.
    Similarly to the proof of Theorem \ref{thm:main_af}, computing the coefficients by the definition of $\sS_{z,k}^{\rT}$, we have
    \begin{align}
        C(w_{z,1}\ast\sS_{z}(x)^{\ast 2};x^{2N})
        &=\sum_{i=0}^{N}\frac{w_{z,1}\ast w_{z,i}\ast w_{z,N-i-1}}{(2i+1)!(2N-2i-1)!},\\
        C(\sC_{z}(x)^{\ast 2};x^{2N})
        &=\sum_{i=0}^{N}\frac{w_{z,i}\ast w_{z,N-i}}{(2i)!(2N-2i)!}
    \end{align}
    for $N\ge 0$.
    We understand the term where $i=N$ of the first sum as $0$.
    Hence we have 
    \begin{align}
        &C(P_{z};x^{2N})\\
        &=\sum_{i=0}^{N}\left(\frac{w_{z,i}\ast w_{z,N-i}}{(2i)!(2N-2i)!}-\frac{w_{z,1}\ast w_{z,i}\ast w_{z,N-i-1}}{(2i+1)!(2N-2i-1)!}\right)\\
        &=\begin{multlined}[t]
            \frac{1}{(2N)!}\sum_{i=0}^{N}\left(\binom{2N}{2i}(w_{z,N}-g_{z,i,N-i})\right.\\
            \left.-\binom{2N}{2i+1}(w_{z,N}-g_{z,1,N-1}-w_{z,1}\ast g_{z,i,N-i-1})\right)
        \end{multlined}\\
        &=\begin{multlined}[t]
            \frac{w_{z,N}}{(2N)!}\sum_{i=0}^{N}\left(\binom{2N}{2i}-\binom{2N}{2i+1}\right)\\
            +\frac{1}{(2N)!}\sum_{i=0}^{N}\left(-\binom{2N}{2i}g_{z,i,N-i}-\binom{2N}{2i+1}(-g_{z,1,N-1}-w_{z,1}\ast g_{z,i,N-i-1})\right)
        \end{multlined}\\
        &=\begin{multlined}[t]
            \frac{w_{z,N}}{(2N)!}\delta_{N,0}\\+\frac{1}{(2N)!}\sum_{i=0}^{N}\left(-\binom{2N}{2i}g_{z,i,N-i}-\binom{2N}{2i+1}(-g_{z,1,N-1}-w_{z,1}\ast g_{z,i,N-i-1})\right).
        \end{multlined}
    \end{align}
    As the last sum is equal to $0$ modulo $\AP_{z}$, the statement \eqref{pt1} derives \eqref{pt2}.\\

    Let us prove the rest implication \eqref{pt2} $\implies$ \eqref{pt1}.
    Since $\AP_{z}$ is generated by $w_{z,n}-w_{z,1}^{\ast n}$ with $n\ge 1$, by induction on $n$, we prove $w_{z,n}=w_{z,1}^{\ast n}$ in a quotient ring $\fH_{M,\ast}/I$ where the ideal $I$ contains each $C(P_{z};x^{2N})$. 
    The $n=1$ case is clear.
    Assume that there exists $n\ge 1$ such that $w_{z,i}-w_{z,1}^{\ast i}$ for every $1\le i\le n$ in $\fH_{M,\ast}/I$.
    By the previous computation, we obtain 
    \begin{align}
        &(2n+2)!C(P_{z};x^{2n+2})\\
        &=\sum_{i=0}^{n+1}\left(\binom{2n+2}{2i}w_{z,i}\ast w_{z,n+1-i}-\binom{2n+2}{2i+1}w_{z,1}\ast w_{z,i}\ast w_{z,n-i}\right)\\
        &=\begin{multlined}[t]
            2w_{z,n+1}-(2n+2)w_{z,1}\ast w_{z,n}\\
            +\sum_{i=1}^{n}\left(\binom{2n+2}{2i}w_{z,i}\ast w_{z,n+1-i}-\binom{2n+2}{2i+1}w_{z,1}\ast w_{z,i}\ast w_{z,n-i}\right).
        \end{multlined}
    \end{align}
    Using the induction hypothesis, we have $w_{z,1}\ast w_{z,i}\ast w_{z,n-i}=w_{z,i}\ast w_{z,n+1-i}=w_{z,1}^{n+1}$ for $1\le i\le n$ in $\fH_{M,\ast}/I$.
    Therefore it follows that
    \begin{align}
        &(2n+2)!C(P_{z};x^{2n+2})\\
        &=2w_{z,n+1}-(2n+2)w_{z,1}^{\ast (n+1)}+\sum_{i=1}^{n}\left(\binom{2n+2}{2i}-\binom{2n+2}{2i+1}\right)w_{z,1}^{\ast (n+1)}\\
        &=2w_{z,n+1}-2w_{z,1}^{\ast(n+1)}.
    \end{align}
    Since the right-hand side is $0$ in $\fH_{M,\ast}/I$, we see that \eqref{pt1} follows from \eqref{pt2}.
\end{proof}
Combining Theorems \ref{thm:main_af} and \ref{thm:main_pt}, we obtain Theorem \ref{thm:main}.
\section{Recovering classical formulas}\label{sec:ii}
In this section, we construct the homomorphism $Z$ mentioned in Remark \ref{rem:assumption} from iterated integrals.
Consider $M\coloneqq\{z\in\bbC\mid |z|\ge 1\}\cup\{0\}$ and endow it with the usual product induced from $\bbC$.
We define a subspace $\fH_{M}^{0}$ of $\fH_{M}$ by
\[\fH_{M}^{0}\coloneqq\bbQ\oplus\bigoplus_{a\in M\setminus\{0,1\}}\fH_{M}e_{a}\oplus\bigoplus_{\substack{a\in M\setminus\{0\}\\ b\in M\setminus\{1\}}}e_{a}\fH_{M}e_{b}.\]
This subspace is closed about the harmonic product.
Denote by $\fH_{M,\ast}^{0}$ it when we consider it as a $\bbQ$-subalgebra of $\fH_{M,\ast}$.
For a word (monic monomial) $w=e_{a_{1}}\cdots e_{a_{k}}$ in $\fH_{M}^{0}$, we put $\wt(w)\coloneqq k$. 
\begin{definition}
    Let $s,u$ be real numbers satisfying $0\le s<u\le 1$.
    For a word $w=e_{z_{1}}\cdots e_{z_{k}}$ of $\fH_{M}^{0}$, we put
    \[I_{s,u}(w)\coloneqq\int_{s<t_{1}<\cdots<t_{k}<u}\prod_{i=1}^{k}\frac{dt_{i}}{t_{i}-z_{i}}.\]
    When $k=0$, the right-hand side is regarded as $1$.
    Extending it $\bbQ$-linearly, we define $I_{s,u}\colon\fH_{M}^{0}\to\bbC$.
\end{definition}
\begin{theorem}\label{thm:harmonic}
    The map $I_{0,1}\colon\fH_{M,\ast}^{0}\to\bbC$ is a $\bbQ$-algebra homomorphism.
\end{theorem}
    \begin{proof}
    For $0<u<1$ and words $w$ and $w'=e_{b_{1}}\cdots e_{b_{k}}$ of $\fH_{M}^{0}$, we put
    \[L_{u}(w;w')\coloneqq\int_{u<t_{1}<\cdots<t_{k}<1}I_{u/t_{1},1}(w)\prod_{j=1}^{k}\frac{dt_{j}}{t_{j}-b_{j}}.\]
    Since $L_{u}(1;w')=I_{u,1}(w')$ holds for a word $w'$, we understand $L_{u}(w;w')$ as $1$ when $w=w'=1$. 
    We shall prove $L_{u}(w;w')=I_{u,1}(w\ast w')$ for non-empty words $w$ and $w'$ of $\fH_{M}^{0}$.
    Then the desired statement follows by the $\bbQ$-linearity, the limit $u\to +0$ and the definition of the harmonic product.
    We execute the proof by induction on $\ell=\wt(w)+\wt(w')$.
    When $\ell=0$ it is clear that $L_{u}(w;w')=I_{u,1}(w\ast w')$ holds.
    Assume there exists an integer $\ell\ge 1$ such that $L_{u}(A;B)=I_{u,1}(A\ast B)$ holds for any words $A$ and $B$ with $\wt(A)+\wt(B)=\ell$. 
    Take $a,b\in M$ and a word $w$ and $w'=e_{b_{1}}\cdots e_{b_{k}}$ such that $\wt(w)+k=\ell-1$.
    Then we have
    \begin{align}
    L_{u}(e_{a}w;e_{b}w')
    &=\int_{u<t_{0}<\cdots<t_{k}<1}I_{u/t_{0},1}(e_{a}w)\frac{dt_{0}}{t_{0}-b}\prod_{j=1}^{k}\frac{dt_{j}}{t_{j}-b_{j}}\\
    &=\int_{\substack{u<t_{0}<\cdots<t_{k}<1\\ u/t_{0}<s<1}}I_{s,1}(w)\frac{ds}{s-a}\frac{dt_{0}}{t_{0}-b}\prod_{j=1}^{k}\frac{dt_{j}}{t_{j}-b_{j}}.
    \end{align}
    Applying the partial fraction decomposition
    \[\frac{1}{s-a}\frac{1}{t_{0}-b}=\frac{1}{st_{0}-ab}\left(\frac{t_{0}}{t_{0}-b}+\frac{a}{s-a}\right),\]
    we obtain
    \begin{multline}
    L_{u}(e_{a}w;e_{b}w')
    =\int_{\substack{u<t_{0}<\cdots<t_{k}<1\\ u/t_{0}<s<1}}I_{s,1}(w)\frac{ds}{st_{0}-ab}\frac{t_{0}\,dt_{0}}{t_{0}-b}\prod_{j=1}^{k}\frac{dt_{j}}{t_{j}-b_{j}}\\
    +\int_{\substack{u<t_{0}<\cdots<t_{k}<1\\ u/t_{0}<s<1}}I_{s,1}(w)\frac{1}{st_{0}-ab}\frac{a\,dsdt_{0}}{s-a}\prod_{j=1}^{k}\frac{dt_{j}}{t_{j}-b_{j}}.
    \end{multline}
    By the substitution $st_{0}=v$ in the first term of the right-hand side, it follows that
    \begin{align}
    &\int_{\substack{u<t_{0}<\cdots<t_{k}<1\\ u/t_{0}<s<1}}I_{s,1}(w)\frac{ds}{st_{0}-ab}\frac{t_{0}\,dt_{0}}{t_{0}-b}\prod_{j=1}^{k}\frac{dt_{j}}{t_{j}-b_{j}}\\
    &=\int_{u<v<t_{0}<\cdots<t_{k}<1}I_{v/t_{0},1}(w)\frac{dv}{v-ab}\frac{dt_{0}}{t_{0}-b}\prod_{j=1}^{k}\frac{dt_{j}}{t_{j}-b_{j}}\\
    &=\int_{u<v<1}L_{v}(w;e_{b}w')\frac{dv}{v-ab}.
    \end{align} 
    This is equal to $I_{u,1}(e_{ab}(w\ast e_{b}w'))$ by the induction hypothesis.
    Thus it suffices to show
    \begin{multline}
        I_{u,1}\left(e_{ab}(e_{a}w\ast w'-e_{0}(w\ast w'))\right)\\
        =\int_{\substack{u<t_{0}<\cdots <t_{k}<1\\ u/t_{0}<s<1}}I_{s,1}(w)\frac{1}{st_{0}-ab}\frac{a\,ds}{s-a}\prod_{j=1}^{k}\frac{dt_{j}}{t_{j}-b_{j}}\,dt_{0}.
\end{multline}
    Note that $e_{ab}(e_{a}w\ast w'-e_{0}(w\ast w'))=e_{ab}((e_{a}-e_{0})w\ast w')$ holds by virtue of Lemma \ref{lem:harmonic_e0}.
    Using the induction hypothesis again, the left-hand side is computed as
    \begin{align}
    &I_{u,1}\left(e_{ab}((e_{a}-e_{0})w\ast w')\right)\\
    &=\int_{u<v<1}L_{v}((e_{a}-e_{0})w;w')\frac{dv}{v-ab}\\
    &=\int_{u<v<t_{1}<\cdots<t_{k}<1}\frac{1}{v-ab}I_{v/t_{1},1}((e_{a}-e_{0})w)\,dv\prod_{j=1}^{k}\frac{dt_{j}}{t_{j}-b_{j}}\\
    &=\int_{\substack{u<v<t_{1}<\cdots<t_{k}<1\\ v/t_{1}<s<1}}\frac{dv}{v-ab}\left(\frac{1}{s-a}-\frac{1}{s}\right)I_{s,1}(w)\,ds\prod_{j=1}^{k}\frac{dt_{j}}{t_{j}-b_{j}}.
    \end{align}
    Therefore we should show
    \begin{multline}\label{eq:star1}
        \int_{\substack{u<v<t_{1}<\cdots<t_{k}<1\\ v/t_{1}<s<1}}\frac{dv}{v-ab}\left(\frac{1}{s-a}-\frac{1}{s}\right)I_{s,1}(w)\,ds\prod_{j=1}^{k}\frac{dt_{j}}{t_{j}-b_{j}}\\
        =\int_{\substack{u<t_{0}<\cdots <t_{k}<1\\ u/t_{0}<s<1}}I_{s,1}(w)\frac{1}{st_{0}-ab}\frac{a\,ds}{s-a}\prod_{j=1}^{k}\frac{dt_{j}}{t_{j}-b_{j}}\,dt_{0}.
    \end{multline}
    On the other hand, we see that the identity
    \begin{multline}\label{eq:star2}
        \int_{\substack{u<v<t_{1}\\ v/t_{1}<s<1}}\frac{dv}{v-ab}\left(\frac{1}{s-a}-\frac{1}{s}\right)I_{s,1}(w)\,ds\\
        =\int_{\substack{u<t_{0}<t_{1}\\ u/t_{0}<s<1}}I_{s,1}(w)\frac{1}{st_{0}-ab}\frac{a\,ds}{s-a}\,dt_{0}
    \end{multline}
    holds for $t_{1}\in(0,u)$ by the substitution $v=st_{0}$ in the left-hand side.
    Since \eqref{eq:star2} holds, we have \eqref{eq:star1}, which is the desired formula.
    \end{proof}
    Next, we define the \emph{harmonic regularization} of iterated integrals.
    Define $\fH_{M}^{1}\coloneqq\bigoplus_{m=0}^{\infty}\fH_{M}^{0}e_{1}^{m}$ and denote by $\fH_{M,\ast}^{1}$ it equipped with $\ast$.
    For a word $w=e_{a_{1}}\cdots e_{a_{k}}$ in $\fH_{M}^{i}$ ($i\in\{0,1\}$), put $l(w)\coloneqq |\{i\mid a_{i}\neq 0\}|$.
    We define $\ell(w)\coloneqq 0$ for $k=1$.
    Furthermore, for a non-negative integer $d$, let $\cF_{d}\fH_{M}^{i}$ be the subspace of $\fH_{M}^{i}$ defined by
    \[\cF_{d}\fH_{M}^{i}\coloneqq\mathrm{span}_{\bbQ}\{w\in\fH_{M}^{i}\text{: word}\mid \ell(w)\le d\}\]
    \begin{lemma}\label{lem:aux1}
        Let $d$ be a positive integer and $w$ a word in $\cF_{d}\fH_{M}^{1}$.
        Then there exist $v_{1}\in\cF_{d}\fH_{M}^{0}$ and $v_{2},v_{3}\in\cF_{d-1}\fH_{M}^{0}$ such that $w=v_{1}+v_{2}\ast e_{1}+v_{3}$.
    \end{lemma}
    \begin{proof}
        Fix $d$ and put $\cG_{m}\fH_{M}^{1}\coloneqq\bigoplus_{i=0}^{m}\fH_{M}^{0}e_{1}^{i}$ for a non-negative integer $m$ (regard it as $\{0\}$ when $m=-1$).
        For a word $w$ in $\fH_{M}^{1}$, define $\nu(w)$ as the unique non-negative integer $m$ such that $w$ represents an element of the quotient space $\cG_{m}\fH_{M}^{1}/\cG_{m-1}\fH_{M}^{1}$.
        We prove the lemma by induction on $\nu(w)$.
        As the $\nu(w)=0$ case is obvious by putting $v_{1}=w$ and $v_{2}=v_{3}=0$, we assume that there exists $m\ge 1$ such that the assertion holds for every $v$ with $\nu(v)=m-1$.
        Let $w$ be a word with $\nu(w)=m$ and write $w=w'e_{1}^{m}$.
        By \eqref{eq:usual_harmonic}, there exists $v_{0}\in\cF_{d-1}\fH_{M}^{1}$ such that
        \begin{equation}\label{eq:aux1}
            w'e_{1}^{m-1}\ast e_{1}=mw+\sum_{i=1}^{k}e_{a_{1}}\cdots e_{a_{i-1}}e_{1}e_{a_{i}}\cdots e_{a_{k}}e_{1}^{m-1}+v_{0},
    \end{equation}
        where $w'=e_{a_{1}}\cdots e_{a_{k}}$.
        By the induction hypothesis, there exist $v'_{1}\in\cF_{d}\fH_{M}^{0}$ and $v'_{2},v'_{3}\in\cF_{d-1}\fH_{M}^{1}$ such that
        \[\sum_{i=1}^{k}e_{a_{1}}\cdots e_{a_{i-1}}e_{1}e_{a_{i}}\cdots e_{a_{k}}e_{1}^{m-1}=v'_{1}+v'_{2}\ast e_{1}+v'_{3}\]
        holds.
        Applying this to \eqref{eq:aux1}, we obtain
        \[w=\frac{1}{m}\left(w'e_{1}^{m-1}\ast e_{1}-v'_{1}-v'_{2}\ast e_{1}-v'_{3}-v_{0}\right).\]
        Hence putting $v_{1}\coloneqq -v'_{1}$, $v_{2}\coloneqq (w'e_{1}^{m-1}-v'_{2})/m$ and $v_{3}\coloneqq -(v'_{3}+v_{0})/m$ gives the desired elements.
    \end{proof}
    \begin{lemma}\label{lem:aux2}
        Let $\phi$ be the $\bbQ$-algebra homomorphism $\fH_{M,\ast}^{0}[S,T]\to\fH_{M,\ast}$ determined by $S\mapsto e_{0}$, $T\mapsto e_{1}$.
        Then $\phi$ is an isomorphism.
    \end{lemma} 
    \begin{proof}
        We divide the proof to two parts: we show that
        \begin{enumerate}
            \item\label{iso1} the $\bbQ$-algebra homomorphism $\phi_{1}\colon\fH_{M,\ast}^{1}[S]\to\fH_{M,\ast}$ determined by $S\mapsto e_{0}$ is an isomorphism.
            \item\label{iso2} the $\bbQ$-algebra homomorphism $\phi_{0}\colon\fH_{M,\ast}^{0}[T]\to\fH_{M,\ast}^{1}$ determined by $T\mapsto e_{1}$ is an isomorphism.
        \end{enumerate}
        First we prove the injectivity of $\phi_{1}$.
        Let $f=\sum_{i=0}^{N}w_{i}S^{i}$ ($w_{i}\in\fH_{M}^{1}$) be an arbitrary element of $\fH_{M,\ast}^{1}[S]$.
        Then, from Lemma \ref{lem:harmonic_e0}, we have
        \[\phi_{1}(f)=\sum_{i=0}^{N}w_{i}\ast e_{0}^{\ast i}=\sum_{i=0}^{N}e_{0}^{i}w_{i}\]
        and the right-hand side never be $0$ unless every $w_{i}$ is $0$ because $w_{i}\in\fH_{M}^{1}$.
        The surjectivity of $\phi_{1}$ is immediately follows from Lemma \ref{lem:harmonic_e0}.\\

        Next we prove the injectivity of $\phi_{0}$.
        Recall the definition of $\cG_{m}$ from the proof of Lemma \ref{lem:aux1}.
        Let $f$ be an element of $\fH_{M,\ast}^{0}[T]$, $d$ the degree of $f$ as a polynomial of $T$ and $w_{1}\neq 0$ the coefficient of $f$ at $T^{d}$.
        We prove $\phi_{0}(f)\neq 0$ when $d\ge 1$ because the injectivity for $d=0$ is easy.
        Then, since both $e_{1}^{\ast d}-d!e_{1}^{d}$ and $\phi_{0}(f-w_{1}T^{d})$ are in $\cG_{m-1}\fH_{M}^{1}$, we have $\phi_{0}(f)-d!w_{1}e_{1}^{d}\in\cG_{d-1}\fH_{M}^{1}$.
        Thus the term $d!w_{1}e_{1}^{d}$ appearing in $\phi_{0}(f)$ never be canceled and hence $\phi_{0}(f)\neq 0$.\\

        Finally, let us show the surjectivity of $\phi_{0}$.
        It is sufficient to prove that every word in $w\in\cF_{d}\fH_{M}^{1}$ is written as a polynomial of $e_{1}$ about $\ast$ for any $d\ge 1$.
        This follows from induction on $d$: when it holds for some $d-1$, any $w\in\cF_{d}\fH_{M}^{1}$ has a decomposition $w=v_{1}+v_{2}\ast e_{1}+v_{3}$ as in Lemma \ref{lem:aux1} and each of $v_{2}$ and $v_{3}$ is written as a polynomial of $e_{1}$ by the induction hypothesis.
    \end{proof}
    By Lemma \ref{lem:aux2}, we obtain the unique $\bbQ$-algebra homomorphism $Z_{S,T}\colon\fH_{M,\ast}\to\bbC[S,T]$ which behaves identically on $\fH_{M,\ast}^{0}$ and sends $e_{0}$ to $S$ and $e_{1}$ to $T$.
    We define $Z\coloneqq Z_{0,0}$.
    Then this map satisfies every requirement in Remark \ref{rem:assumption}.
    Indeed, this is a $\bbQ$-algebra homomorphism $\fH_{M,\ast}\to\bbC$ by its construction.
    The first condition \eqref{ass1} also follows from the definition and \eqref{ass2} is obtained by a naive series expansion of the iterated integral.
    The last requirement is well-known in terms of multiple zeta values (e.g.~\cite[Corollary 2.3]{hoffman97}).
    See also Remark \ref{rem:associator}
    \begin{remark}
        As mentioned in Section \ref{sec:introduction}, the iterated integral $I_{0,1}(s_{1,k_{1}}\cdots s_{1,k_{r}})$ coincides with the multiple zeta value $(-1)^{r}\zeta(k_{1},\ldots,k_{r})$.
        For that reason, comparing coefficients and using the harmonic relation (Theorem \ref{thm:harmonic}), we can consider each of the addition formula and Pythagorean theorem as a family of linear relations of multiple zeta values (actually they are equivalent as seen in Theorem \ref{thm:main}).
        For instance, taking the coefficient at $x^{2}$ in the Pythagorean theorem $Z(P_{1}(x))=1$, we obtain $4\zeta(2,2)=3\zeta(4)$.
    \end{remark}
    \begin{remark}\label{rem:associator}
        Let $\bbK$ be a field of characteristic $0$ and $c$ a $\bbQ$-algebra homomorphism $c\colon\bbQ\langle e_{0},e_{1}\rangle\cap\fH_{M,\ast}^{0}\to\bbK$ and extend it to the power series ring.
        By Theorem \ref{thm:main}, if there exists a constant $\mu$ such that $c(w_{1,n})=\mu^{2n}/(2n+1)!$ for any $n\ge 0$, the addition formula and Pythagorean theorem holds for $c(\sS_{1}(x))$.
        Define
        \[\Phi(X_{0},X_{1})\coloneqq\sum_{n=0}^{\infty}\sum_{\substack{a_{1},\ldots,a_{n}\in\{0,1\}}}c(e_{a_{1}}\cdots e_{a_{n}})X_{a_{n}}\cdots X_{a_{1}}\in\bbK\ncjump{X_{0},X_{1}}.\]
        As proved in \cite[Theorem 1]{furusho10}, if $\Phi$ is a commutator group-like series and satisfies the \emph{$5$-cycle relation}, there exists $\mu\in\overline{\bbK}$ (the algebraic closure of $\bbK$) such that $(\Phi,\mu)$ satisfies the \emph{$3$-cycle relation}.
        Moreover, under the same condition, it is showed that $c$ satisfies the \emph{regularized double shuffle relation} in \cite[Theorem 0.2]{furusho11}.
        Since combining the $3$-cycle relation for $(\Phi,\mu)$ and regularized double shuffle relation derives $c(w_{1,n})=\mu^{2n}/(2n+1)!$ (see \cite[Theorem 4.5, Proposition 4.6]{li10}), the special case of the addition formula and Pythagorean theorem follows from the group-likeness and $5$-cycle relation.
        The map $Z$ we constructed above is one such example with $\mu=\pm 2\pi i$ and $\bbK=\overline{\bbK}=\bbC$ as shown in \cite{drinfeld91}. 
        Note that, under these assumptions, this fact can be also proved in a more direct way using \cite[Proposition 4.6]{li10} instead of Theorem \ref{thm:main}, which states that 
        \[c(\sS_{1}(x))=\frac{e^{\mu x}-e^{-\mu x}}{\mu}.\]
    \end{remark}
    \section*{Acknowledgement}
The author would like to thank Prof.~Shin-ichiro Seki for many helpful advices and encouraging him to write this paper.
The author is also grateful to Prof.~Masataka Ono and Yuno Suzuki for their careful reading of the manuscript.\\
    \textbf{Declarations of interest}: none.

\end{document}